\documentclass{amsart}
\usepackage{amssymb,amsmath}
\newcommand{\RR}{{\mathbb R}}

\newcommand{\e}{\varpi}

\newcommand{\de}{\delta}
\newcommand{\del}{\partial}
\newcommand{\om}{\omega}

\newcommand{{\loc}}{{\ell\mathrm oc}}

\def\meanint{{\diagup\hskip -.42cm\int}}

\newtheorem{theorem}{Theorem}

\newtheorem{proposition}{Proposition}
\newtheorem{corollary}{Corollary}

\begin{document}

\title{ Gradient estimate for  solutions of second-order elliptic equations}
\author{Vladimir Maz'ya}
\address{Department of Mathematics, Link\"oping University, SE-581 83 Link\"oping, Sweden and RUDN, 6 Miklukho-Maklay St, Moscow, 117198, Russia}
\email{vladimir.mazya@liu.se}

\author{ Robert McOwen}
\address{Department of Mathematics, Northeastern University, Boston, MA 02115}
\email{r.mcowen@northeastern.edu}
\date{November 22, 2021}


\begin{abstract}
We obtain a local estimate for the gradient of solutions to a second-order elliptic equation in divergence form with bounded measurable coefficients that are square-Dini continuous at the single point $x=0$. In particular, we treat the case of solutions that are not Lipschitz continuous at $x=0$. We show that our  estimate is sharp.

\end{abstract}
\maketitle

\section{Introduction}\label{sec:Intro}

We consider a uniformly elliptic equation in divergence form:
\begin{equation}\label{Lu=0}
\partial_j (a_{ij}(x)\,\partial_i u)=0 \quad\hbox{for $x\in B_1$},
 \end{equation}
 where we have used the summation convention and the coefficients $a_{ij}=a_{ji}$ are real-valued bounded measurable functions in the ball $B_{1}:=\{x\in \RR^n:|x|<1\}$, $n\geq 2$.  
  A weak solution $u\in H^{1,2}(B_1)$, i.e.\ $u,\nabla u\in L^2(B_1)$, is known to be H\"older continuous by De Giorgi \cite{DG} and Nash \cite{N}, but may have unbounded gradient. Under additional regularity of the $a_{ij}$ such as Dini continuity or Dini mean oscillation in a domain, weak solutions are known to be $C^1$:  cf.\ \cite{HW} and \cite{DK}. Gradient bounds in a domain have also been obtained for nonlinear problems,
including  minimal regularity on the data: cf.\ \cite{Mi}, \cite{KuM}, \cite{CM2}, \cite{CM3}, and \cite{DM}.  
   In \cite{MM} we obtained conditions on the coefficients $a_{ij}$ just at $x=0$  which imply that all weak solutions of \eqref{Lu=0} are Lipschitz continuous at $x=0$. In this paper we allow the possibility of weak solutions that are not Lipschitz continuous at $x=0$, but obtain bounds on the growth of their gradient in mean value. We are not aware of any previous results in the literature of this type.

Let us describe our assumptions on the coefficients $a_{ij}(x)$. These are most conveniently stated when $a_{ij}(0)=\delta_{ij}$, which can be achieved by a change of coordinates. We want the $a_{ij}(x)$ to be {\it square-Dini continuous at $x=0$}, which means
\begin{equation}\label{a_ij-delta_ij}
\sup_{|x|=r}|a_{ij}(x)-\delta_{ij}|\leq \om(r) \quad\hbox{as}\ r=|x|\to 0,
\end{equation}
where the modulus of continuity $\om(r)$ is a nondecreasing, positive, continuous function for $r>0$ satisfying
\begin{equation}\label{square-Dini}
\int_0^{1} \frac{\om^2(r)}{r}\,dr<\infty.
\end{equation}
Since $\om(r)$ does not vanish rapidly as $r\to 0$, it is natural and often convenient to assume
\begin{equation}\label{om-notrapiddecrease}
\hbox{for some $\kappa>0$}, \ \om(r)\,r^{-1+\kappa} \ \hbox{is nonincreasing for $0<r<1$.}
\end{equation}

As seen in  \cite{MM}, the conditions \eqref{a_ij-delta_ij}-\eqref{square-Dini} allow
the regularity of solutions of \eqref{Lu=0} at $x=0$ to be determined by the asymptotic behavior of solutions of the $n$-dimensional dynamical system
\begin{equation}\label{R-DynSys}
\frac{d\phi}{dt}+R(e^{-t})\phi=0 \quad\hbox{for}\ 0<t<\infty,
\end{equation}
where $R(r)$ is the  matrix function
 \begin{equation}\label{def:R}
 R(r):=\meanint_{S^{n-1}}\left(A(r\theta)-nA(r\theta)\theta\otimes\theta\right)\,d\theta \quad\hbox{for $0<r<1$.}
\end{equation}
 Here $A=(a_{ij})$, $A\theta\otimes\theta$ is the outer product of the vectors $A\theta$ and $\theta$, and $d\theta$ denotes standard surface measure on $S^{n-1}$; the slashed integral in \eqref{def:R} and throughout this paper denotes mean value.
We note that\footnote{Throughout this paper we use $c$ to denote a constant that does not depend upon the solution $u$ but whose value may change with each occurence.} $|R(r)|\leq c\,\om(r)$, but  the matrix $R$ need not be symmetric.
In  \cite{MM} we found that, if the dynamical system \eqref{R-DynSys} is uniformly stable as $t\to\infty$, then a weak solution of \eqref{Lu=0} must be Lipschitz continuous at $x=0$. In this paper, we do not assume \eqref{R-DynSys} is uniformly stable, but seek an estimate on the growth of the gradient of a weak solution.

To formulate our main result, for a square matrix $M$ let $\mu[M]$ denote the largest eigenvalue of the symmetric matrix $(M+M^*)/2$, where $M^*$ denotes the adjoint (i.e.\ transpose) of $M$. 
\begin{theorem}\label{Main Theorem} 
Under the above conditions on $a_{ij}$, if $u\in H^{1,2}(B_1)$ is a weak solution of \eqref{Lu=0}, then
\begin{equation}\label{grad(u)-est}
\left(\meanint_{r<|x|<2r} |\nabla u(x)|^2\,dx\right)^{1/2} \leq c\,\|u\|_{L^2(B_1)}\,\exp\left( \int_r^1 \mu[-R(\rho)]\,\frac{d\rho}{\rho}\right) \quad\hbox{for}\ 0<r<1/2.
\end{equation}
\end{theorem}
\noindent
We will give an example in Section \ref{sec:GS} that shows the estimate \eqref{grad(u)-est} is sharp.
The proof of  \eqref{grad(u)-est} will  only concern $0<r<r_0$ for $r_0$ sufficiently small, since  \eqref{grad(u)-est}  is clearly true for $r_0<r<1$.
Under additional assumptions on the $a_{ij}(x)$ for $x\not= 0$, it is possible to replace the left hand side of \eqref{grad(u)-est} by a local pointwise estimate; see, for example, Theorem 2 in \cite{MM2}.  While our motivation for this study was to estimate the growth of the gradient of solutions, Theorem \ref{Main Theorem} also applies to decay. In particular, if
\begin{equation}\label{int(mu)->-infty}
 \int_r^1 \mu[-R(\rho)]\,\frac{d\rho}{\rho}\to -\infty \quad\hbox{as}\ r\to 0,
\end{equation}
then the $L^2$-mean of the gradient of solutions must vanish as $r\to 0$. In fact, in \cite{MM}  it was shown that \eqref{int(mu)->-infty} implies the stronger conclusion that $u$ is differentiable at $x=0$ and $\nabla u(0)=0$.

Let us make some additional comments. First of all, let us denote our estimator by
 \begin{equation}\label{def:E}
 E(r):= \exp\left[\int_r^1  \mu[-R(\rho)] \,\frac{d\rho}{\rho}\right].
 \end{equation}
If $\mu[-R(r)]\geq 0$, then $E(r)$ is nonincreasing in $r$, and we may have $E(r)\to\infty$ as $r\to 0$, but not quickly since $\mu[-R(r)]\to 0$ as $r\to 0$. 
Similarly, if $\mu[-R(r)]\leq 0$, then $E(r)$ is nondecreasing in $r$, and when \eqref{int(mu)->-infty} occurs we  have $E(r)\to 0$ as $r\to 0$, but not quickly. 
 In fact, for any $0<\lambda<1$, we can check that $E(r)\,r^{-\lambda}$ is decreasing and $E(r)\,r^{-\lambda}$ is increasing on $0<r<r_0$ for $r_0$ sufficiently small. Since we could replace $B_1$ in \eqref{Lu=0} by $B_{r_0}$, we will simply assume
  \begin{equation}\label{E*r^lambda-decreasing}
  E(r)\,r^{-\lambda} \ \hbox{is decreasing and} \ 
 E(r)\,r^{\lambda} \ \hbox{is increasing  for $0<r<1$.}
   \end{equation}

Secondly, regarding the proof, we will write the weak solution $u$ in the following form:
\begin{subequations}
\begin{equation}\label{u-asym}
u(x)=u_0(|x|) + \vec v(|x|)\cdot x + w(x) \quad\hbox{for}\ 0<|x|<1,
\end{equation}
where 
\begin{equation}\label{u0,vk=}
u_0(r):=\meanint_{S^{n-1}} u(r\theta)\,d\theta, \quad v_k(r)=\frac{n}{r}\meanint_{S^{n-1}} u(r\theta)\,\theta_k\,d\theta, \quad \theta_k=x_k/|x|,
\end{equation}
and 
\begin{equation}\label{int-w=0}
\meanint_{S^{n-1}} w(r\theta)\,d\theta = 0 = \meanint_{S^{n-1}} w(r\theta)\,\theta_k\,d\theta \quad\hbox{for}\ k=1,\dots,n.
\end{equation}
\end{subequations}
We shall show that the functions $\vec v$ and $\vec v\,'$ satisfy a $2n$-dimensional dynamical system that depends on $w$, and that $w$ satisfies a Poisson equation that depends on $\vec v,\vec v\,'$; this enables us to obtain estimates on $\vec v$, $\vec v\,'$, and $ \nabla w$, and then we are able to estimate $u_0'$.
In the estimates for $w$, we will use $L^p$-means over annuli: for $f\in H^{1,p}_\loc(\RR^n\backslash\{0\})$ where $1\leq p<\infty$, let us define
\begin{equation}\label{def:Mp}
M_{p}(f,r) :=\left(\meanint_{A_r} |f(x)|^p\,dx\right)^{1/p}, \quad\hbox{where}\ A_r=\{x\in\RR^n:r<|x|<2r\}.
\end{equation}
The $f$ in \eqref{def:Mp} may be scalar or vector-valued. 
Using this notation, \eqref{grad(u)-est} may be written
\begin{equation}\label{grad(u)-Lp-est}
M_{2}(\nabla u,r) \leq c\,E(r)\,\|u\|_{L^{2}(B_1)} \quad\hbox{for} \ 0<r<1/2.
\end{equation}

  \section{Estimates for an $n$-dimensional Dynamical System}\label{sec:n-dim}
  
  If we let $R_1(t)=R(e^{-t})$ for $R(t)$ as in \eqref{def:R} and $\e(t)=\om(e^{-t})$, then we can write the dynamical system \eqref{R-DynSys} as 
\begin{equation}\label{DynSys-R1}
\frac{d\phi}{dt}=-R_1(t)\phi \quad\hbox{for}\  t>0,
\end{equation}
where $|R_1(t)|\leq c\,\e(t)$ for $t>0$.
If we take scalar product with $\phi$ on both sides of \eqref{DynSys-R1}, we obtain
\[
\frac{1}{2}\frac{d}{dt}\,|\phi|^2=-(R_1 \phi,\phi)=-\frac{1}{2}([R_1(t)+R_1^*(t)]\phi,\phi).
\]
Using the variational characterization of the largest eigenvalue of a symmetric matrix, we have
\[
\frac{1}{2}\frac{d}{dt}\,|\phi|^2\leq \mu[-R_1(t)] \,|\phi|^2,
\]
and we can integrate this differential inequality to obtain
\[
|\phi(t)|\leq |\phi(0)|\exp\left(\int_0^t \mu[-R_1(s)]\,ds\right).
\]
If we define
\begin{equation}
{\mathcal E}(t):=E(e^{-t})=\exp\left(\int_0^t  \mu[-R_1(s)]\,ds\right),
\end{equation}
then we have proved the following:
\begin{proposition}\label{pr:1} 
The fundamental matrix $\Phi$ for \eqref{DynSys-R1} with $\Phi(0)=I$ satisfies the following:
\begin{equation}\label{est:Phi}
|\Phi(t)|\leq {\mathcal E}(t) \quad\hbox{for}\ t>0
\end{equation}
and
\begin{equation}\label{est:Phi*Phi^(-1)}
|\Phi(t) \Phi^{-1}(s)|\leq  \exp\left(\int_s^t \mu[-R_1(\upsilon)]\,d\upsilon\right)=\frac{{\mathcal E}(t)}{{\mathcal E}(s)}, \quad \hbox{for}\ 0<s<t.
\end{equation}
\end{proposition}

Of course, the fundamental matrix $\Phi$ is useful in solving the nonhomogeneous equation
\begin{equation}\label{NH-R1}
\frac{d\phi}{dt}+R_1(t)\phi=f(t).
\end{equation}
\begin{corollary}\label{cor:1} 
If ${\mathcal E}^{-1}(t)\,f(t)\in L^1(0,\infty)$ then the solution $\phi$ of \eqref{NH-R1} satisfies
\[
|\phi(t)|\leq {\mathcal E}(t) \left(|\phi(0)|+\|{\mathcal E}^{-1}\,f\|_{L^1}\right).
\]
\end{corollary}
\begin{proof}
This follows immediately from \eqref{est:Phi*Phi^(-1)}  and the variation of constants formula:
\[
\phi(t)=\Phi(t)\left(\phi(0)+\int_0^t \Phi^{-1}(s)\,f(s)\,ds\right). 
\]
\end{proof}

  \section{Estimates for a $2n$-dimensional Dynamical System}\label{sec:2n-dim}

In this section we consider the $2n\times 2n$ system on $(0,\infty)$
\begin{subequations} \label{DS:joint}
\begin{equation} \label{DS:system}
\frac{d}{dt}
\begin{pmatrix}
\phi \\ \psi
\end{pmatrix} + \begin{pmatrix}
0 & 0 \\ 0 & -nI
\end{pmatrix}
\begin{pmatrix}
\phi \\ \psi
\end{pmatrix} + {\mathcal R}(t)
\begin{pmatrix}
\phi \\ \psi
\end{pmatrix} =  F(t),
\end{equation}
where i) ${\mathcal R}$ is a $2n\times 2n$ matrix of the form
\begin{equation}\label{DS:R}
{\mathcal R}(t)=
\begin{pmatrix}
R_{1}(t) & R_{2}(t) \\ R_{3}(t) & R_{4}(t)
\end{pmatrix}
\quad
\hbox{with}\ \pmb{\bigr |}R_j(t)\pmb{\bigr |}\leq \e(t)\ \hbox{on}\ 0<t<\infty,
\end{equation}
and ii)  $F=(F_1,F_2)$ satisfies
 \begin{equation}\label{DS:F1}
 {\mathcal E}^{-1}F_1 \in L^1(0,\infty), 
 \end{equation}
 and there exists $\de>0$ so that for any choice of $\alpha\in [n-\de,n)$ there is a constant $c_\alpha$ so that
\begin{equation}\label{DS:F2}
e^{\alpha t}\int_t^\infty|F_2(s)|e^{-\alpha s}\,ds \leq c_\alpha\,\e(t)\,{\mathcal E}(t)\quad\hbox{for}\ 0<t<\infty.  
\end{equation}
\end{subequations}
We want $\psi$ to satisfy the ``finite energy condition''
\begin{equation}\label{psi-finite_energy}
\int_0^\infty (|\psi|^2+|\psi_t|^2)e^{-nt}dt<\infty.
\end{equation}

\begin{proposition}\label{pr:2} 
Under the above hypotheses, all solutions $(\phi,\psi)$ of \eqref{DS:joint} that satisfy \eqref{psi-finite_energy} must satisfy the following estimates:
\begin{subequations}
\begin{equation}\label{pr2:est-phi}
|\phi(t)|\leq c\, {\mathcal E}(t) (c_\alpha+|\phi(0)|+\| {\mathcal E}^{-1}F_1\|_{L^1})
\end{equation}
\begin{equation}\label{pr2:est-psi}
|\psi(t)|\leq c\,\e(t)\, {\mathcal E}(t) (c_\alpha+|\phi(0)|+\| {\mathcal E}^{-1}F_1\|_{L^1}).
\end{equation}
\end{subequations}
\end{proposition}

\begin{proof}
As in the previous section, we let $\Phi$ denote the fundamental solution for $\dot\phi+R_1\phi=0$ on $[0,\infty)$ with $\Phi(0)=I$.
 We need to estimate the fundamental solution $\Psi(t)$ for $\dot\psi+R_4\psi=0$ with $\Psi(0)=I$.  Since $\e(t)\to 0$ as $t\to\infty$ and $\int_0^\infty\e^2(t)\,dt<\infty$, for any $\delta\in (0,1)$ we have $\e(t)<\delta$ and $\int_{t_0}^\infty<\delta$ for $t_0$ sufficiently large. 
Without loss of generality, we will assume that
\begin{equation}\label{small_epsilon}
\e(t)<\delta \quad \hbox{and} \quad \int_0^\infty \e^2(t)\,dt<\delta.
\end{equation}
Note that this $\delta$ can be made small independent of other constants that may appear.
We can now use Gronwall's inequality to conclude
\begin{equation}\label{PsiPsi^(-1)-estimate} 
|\Psi(t) \Psi^{-1}(s)|\leq e^{\delta|t-s|} \quad\hbox{for}\ t,s>0.
\end{equation}

Applying variation of constants in the first equation in \eqref{DS:system}, we can write
\begin{subequations}
\begin{equation}\label{integraleqn:phi=}
\phi(t)=\Phi(t)\phi(0)+\Phi(t)\int_0^t \Phi^{-1}(s) (F_1(s)-R_2(s)\psi(s))\,ds.
\end{equation}
Variation of constants in the second equation in \eqref{DS:system} combined with \eqref{psi-finite_energy} yields
\begin{equation}\label{integraleqn:psi=}
\psi(t)=e^{nt}\Psi(t)\int_t^\infty \Psi^{-1}(s)(R_3(s)\phi(s)-F_2(s))\,e^{-ns}\,ds.
\end{equation}
\end{subequations}
If we plug \eqref{integraleqn:psi=} into \eqref{integraleqn:phi=} we obtain the following equation for $\phi$:
\begin{subequations}\label{eq:phi}
\begin{equation}
\phi(t)+S\phi(t)=\eta_0(t)+\eta_1(t)+\eta_2(t),
\end{equation}
where
\begin{equation}\label{def:S}
\begin{aligned}
S\phi(t)&=-\Phi(t)\int_0^t \Phi^{-1}(s) \,R_2(s) \,e^{ns}\,\Psi(s)\int_s^\infty \Psi^{-1}(\upsilon)\,R_3(\upsilon)\,\phi(\upsilon)\,e^{-nu}\,d\upsilon\,ds, \\
&\eta_0(t)=\Phi(t)\phi(0), \qquad\quad \eta_1(t)=\Phi(t)\int_0^t \Phi^{-1}(s)\,F_1(s)\,ds, \\
\eta_2(t)&=\Phi(t) \int_0^t \Phi^{-1}(s)\,R_2(s)\,e^{ns}\,\Psi(s)\int_s^\infty \Psi^{-1}(\upsilon)\,F_2(\upsilon) \,e^{-n\upsilon}\,d\upsilon\,ds.
\end{aligned}
\end{equation}
\end{subequations}
We want to show \eqref{eq:phi} has a solution $\phi$ and estimate it in terms of $\e$, ${\mathcal E}$, and $F$. 

Let $X=C([0,\infty),\RR^n)$ with the norm 
\begin{equation}\label{def:X}
\|\phi\|_{X}:=\sup\left\{\frac{|\phi(t)|}{{\mathcal E}(t)}:0<t<\infty\right\}.
\end{equation}
Assume $\|\phi\|_X=1$, so $|\phi(t)|\leq {\mathcal E}(t)$. Using this,  \eqref{est:Phi*Phi^(-1)}, and \eqref{PsiPsi^(-1)-estimate}, we obtain the estimate
\begin{equation}\label{est:S(phi)}
\begin{aligned}
|S\phi(t)|& \leq \int_0^t |\Phi(t)\Phi^{-1}(s)\,|R_2(s)| \,e^{(n-\delta)s}\int_s^\infty e^{(\delta-n)\upsilon}|R_3(\upsilon)|{\mathcal E}(\upsilon)\,d\upsilon\,ds \\
& \leq c\, {\mathcal E}(t)\int_0^t  ({\mathcal E}(s))^{-1}
\,\e^2(s) \,e^{(n-\delta)s}\int_s^\infty e^{(\delta-n)\upsilon}{\mathcal E}(\upsilon)\,d\upsilon\,ds.
\end{aligned}
\end{equation}
Regarding the last integral, we can use integration by parts to obtain
\[
\int_s^\infty e^{(\delta-n)\upsilon}{\mathcal E}(\upsilon)\,d\upsilon=\frac{e^{(\delta-n)s}{\mathcal E}(s)}{n-\delta}+\frac{1}{n-\delta}\int_s^\infty e^{(\delta-n)\upsilon}{\mathcal E}(\upsilon)\,\mu[-R_1(\upsilon)]\,d\upsilon.
\]
But, using $|\mu[-R_1(t)]|\leq c\,\e(t)$ and \eqref{small_epsilon}, we can take $\delta$ small enough that 
\[
\int_s^\infty e^{(\delta-n)\upsilon}{\mathcal E}(\upsilon)\,d\upsilon\leq \frac{e^{(\delta-n)s}{\mathcal E}(s)}{n-\delta}+\frac{1}{n-\delta}\int_s^\infty e^{(\delta-n)\upsilon}{\mathcal E}(\upsilon)\,d\upsilon.
\]
By iterating this, we obtain
\begin{equation}\label{int-exp*E}
\int_s^\infty e^{(\delta-n)\upsilon}{\mathcal E}(\upsilon)\,d\upsilon\leq c\, e^{(\delta-n)s}{\mathcal E}(s).
\end{equation}
Using this in \eqref{est:S(phi)} we obtain
\[
|S\phi(t)|\leq c\,{\mathcal E}(t) \int_0^t \e^2(s)\,ds\leq c\,{\mathcal E}(t)\,\delta.
\]
Taking $\delta<c^{-1}$, we indeed have $\|S\|_{X\to X}<1$.
To verify that $\eta_0,\eta_1\in X$ we use Proposition \ref{pr:1} and Corollary \ref{cor:1} to conclude
\[
|\eta_0(t)+\eta_1(t)|\leq {\mathcal E}(t)\left( |\phi(0)|+\|{\mathcal E}^{-1}F_1\|_{L^1}\right).
\]
  To show $\eta_2\in X$, we use  \eqref{est:Phi*Phi^(-1)}, \eqref{int-exp*E}, and \eqref{DS:F2} with $\alpha=n-\delta$ to estimate
 \[
 \begin{aligned}
 |\eta_2(t)| & \leq c\,{\mathcal E}(t)\int_0^t ({\mathcal E}(s))^{-1}\,\e(s)\,e^{(n-\delta)s}\int_s^\infty e^{(\delta-n)\upsilon}\,F_2(\upsilon)\,d\upsilon\,ds \\
& \leq  c\,{\mathcal E}(t)\int_0^t  ({\mathcal E}(s))^{-1}\,c_\alpha\,\e^2(s)\,ds <c\,c_\alpha\,{\mathcal E}(t). 
\end{aligned}
 \]
 We conclude that there is a solution $\phi\in X$ of \eqref{eq:phi} satisfying \eqref{pr2:est-phi}.
Using \eqref{integraleqn:psi=} we find $\psi$ and hence the solution $(\phi,\psi)$ of \eqref{DS:joint}.  
To estimate $\psi$ we use \eqref{DS:F1}, \eqref{integraleqn:psi=}, and \eqref{pr2:est-phi} to obtain 
 \[
 \begin{aligned}
 |\psi(t)| & \leq e^{(n-\delta)t}\int_t^\infty e^{(\delta-n)s}(\e(s)|\phi(s)|+|F_2(s)|)\,ds \\
 & \leq c\,\e(t)\,{\mathcal E}(t)\left(c_\alpha + |\phi(0)|+\|{\mathcal E}^{-1}F_1\|_{L^1} \right),
 \end{aligned}
 \]
which is  \eqref{pr2:est-psi}.
\end{proof}

  \section{Potential Theory Estimates}\label{sec:PotentialThry}
  
  In the proof of Theorem \ref{Main Theorem}, we will encounter the equation $ - \Delta w=g$ in $\RR^n\backslash\{0\}$,
  where $g$ is a distribution with certain orthogonality properties that we will now describe. 
  For $f\in L^1_{\loc}(\RR^n\backslash\{0\})$, let us define
   \begin{equation}\label{Pf=}
  Pf(r\theta):=\meanint_{S^{n-1}} f(r\phi)\,d\phi+n\theta_k\meanint_{S^{n-1}} \phi_k\,f(r\phi)\,d\phi.
  \end{equation}
  Using 
  \[
  \meanint_{S^{n-1}}\theta_k\theta_\ell\,d\theta=\frac{1}{n}\delta_{k\ell} \quad\hbox{for}\ k,\ell=1,\dots,n,
  \]
  we see that $P^2f=Pf$, and  $Pf$ is the projection of $f$ onto the functions on $S^{n-1}$ spanned by $1,\theta_1,\dots,\theta_n$. 
  Let us define
  \begin{equation}\label{def:perp}
  f(r\theta)^\perp:=f(r\theta)-Pf(r\theta).
  \end{equation}
  To control the growth of solutions near $x=0$ we will not only use the $L^p$-mean over annuli as defined in \eqref{def:Mp}, but the following to control first and second derivatives:
    \begin{equation}\label{def:Mk,p}
    \begin{aligned}
    M_{1,p}(f,r):&=rM_p(\nabla w,r)+M_p(w,r) \\
    M_{2,p}(f,r):&=r^2M_p(D^2 w,r)+M_{1,p}(w,r).
    \end{aligned}
      \end{equation}
      
  The two equations that we want to solve are
  \begin{subequations}
   \begin{equation}\label{eq:Laplace(w)=f}
  -\Delta w= f^\perp \quad\hbox{in}\ \RR^n\backslash\{0\} 
  \end{equation}
  and
  \begin{equation}\label{eq:Laplace(w)=div(f)}
  -\Delta w= [\,{\rm div}\vec f\,]^\perp \quad\hbox{in}\ \RR^n\backslash\{0\}.
  \end{equation}
  \end{subequations}
  In both cases, we will obtain $w$ through convolution with the fundamental solution $\Gamma(|x|)$ for $-\Delta$. 
  The following was proved in \cite{MM} (see Propositions 1 and 2).
  
  \begin{proposition}\label{pr:3} 
  Let $n\geq 2$ and $1<p<\infty$.
  \begin{itemize}
  \item[a)] If $f\in L^p_{\loc}(\RR^n\backslash\{0\})$ satisfies 
  \[
  \int_{|x|<1} |f(x)|\,|x|^2\,dx<\infty \quad\hbox{and}\quad \int_{|x|>1} |f(x)|\,|x|^{-n}\,dx <\infty,
  \]
  then convolution by $\Gamma$ defines a solution $w\in H^{2,p}_{\loc}(\RR^n\backslash\{0\})$ of \eqref{eq:Laplace(w)=f} satisfying
  \[
  M_{2,p}(w,r)\leq c\left(r^{-n} \int_0^r M_p(f,\rho)\,\rho^{n+1}\,d\rho + r^2\int_r^\infty M_p(f,\rho)\,\rho^{-1}\,d\rho\right).
  \]
   \item[b)] If $ \vec f=(f_i)$ with $f_i\in L^p_{\loc}(\RR^n\backslash\{0\})$ satisfies 
  \[
  \int_{|x|<1} |\vec f(x)|\,|x|\,dx<\infty \quad\hbox{and}\quad \int_{|x|>1} |\vec f(x)|\,|x|^{-n-1}\,dx <\infty,
  \]
  then convolution by $\Gamma$ defines a solution $w\in H^{1,p}_{\loc}(\RR^n\backslash\{0\})$ of \eqref{eq:Laplace(w)=div(f)} satisfying
  \[
  M_{1,p}(w,r)\leq c\left(r^{-n} \int_0^r M_p(\vec f,\rho)\,\rho^{n}\,d\rho + r^2\int_r^\infty M_p(\vec f,\rho)\,\rho^{-2}\,d\rho\right).
  \]
  \end{itemize}
  \end{proposition}
  \noindent
The proof of this proposition in \cite{MM} uses an expansion of $\Gamma$ in spherical harmonics.
  For example, for $|x|<|y|$, $\Gamma(|x-y|)$ can be written as a convergent series\footnote{We note that the coefficients $a_{k,m}$ were unfortunately omitted in the corresponding formula (29) in \cite{MM}.}
\begin{equation} \label{Gamma-expansion1}
\Gamma(|x-y|)=\sum_{k=0}^\infty \,\frac{|x|^k}{|y|^{n-2+k}}\sum_{m=1}^{N(k)}a_{k,m}\,\varphi_{k,m}\left(\hat{x}\right)\,\varphi_{k,m}\left(\hat{y}\right),
\end{equation}
where $\varphi_{k,m}$ for $m=1,\dots,N(k)$ is an orthonormal basis for the space ${\mathcal H}(k)$ of spherical harmonics of degree $k$, $N(k)={\rm dim}({\mathcal H}(k)$, $\hat x=x/|x|$, 
and $a_{k,m}$ are  certain coefficients. 
      
  \section{Proof of Theorem 1}\label{sec:Proof}

We assume that $u\in H^{1,2}(B_1)$ is a weak solution of \eqref{Lu=0}. If we introduce a smooth function $\chi(r)$ satisfying $\chi(r)=1$ for $0\leq r\leq 1/4$ and $\chi(r)=0$ for $r\geq 1/2$, then
\[
\partial_j(a_{ij}\partial_i(\chi u))=\partial_j(a_{ij}\chi'\theta_i u)+\chi'\theta_ja_{ij}\partial_i u,
\]
where (using the definition of weak solution)
\[
\int_{B_1} \chi'\theta_j\,a_{ij}\partial_i u\,dx=\int_{B_1} a_{ij}\partial_i u\,\partial_j\chi\,dx=0.
\]
To use the results of the previous section, we need a problem defined on $\RR^n$, so let us
 extend $a_{ij}$ and $\om(r)$ to $\RR^n$ by defining 
\begin{equation}\label{a_{ij}_r>1}
a_{ij}(x)=\delta_{ij}\ \hbox{for}\ |x|>1 \quad\hbox{and}\quad \om(r)=\om(1) \ \hbox{for}\ r>1.
\end{equation}
Now let us consider the problem
\begin{subequations}
\begin{equation}\label{Lu=f}
\partial_j (a_{ij}(x)\,\partial_i \widetilde u)={\rm div}\vec f +f_0, \quad\hbox{for $x\in \RR^n$},
 \end{equation}
where 
\begin{equation}\label{def:f,f0}
f_j:=a_{ij}\chi'\theta_i u \quad\hbox{and}\quad f_0:=\chi'\theta_ja_{ij}\partial_i u
 \end{equation}
are supported in $1/4<|x|<1/2$ and satisfy
\begin{equation}\label{est:f,f0}
\|\vec f\|_{L^2}+\|f_0\|_{L^2}\leq c\,\|u\|_{H^{1,2}(B_{1/2})}\leq c\,\|u\|_{L^2(B_1)},
 \end{equation}
  where we have also used the Cacciopolli inequality,
Moreover, as observed above, we have
 \begin{equation}\label{int(f0)=0}
 \int_{\RR^n} f_0(x)\,dx=0.
 \end{equation}
 \end{subequations}
 We want $ \widetilde u\in H^{1,2}(\RR^n)$ to be a weak solution of \eqref{Lu=f} in the form
 \begin{equation}\label{tilde-u-asym}
\widetilde u(x)=u_0(|x|) + \vec v(|x|)\cdot x + w(x) \quad\hbox{for}\ x\in\RR^n\backslash\{0\},
\end{equation}
where the $u_0$, $\vec v$ and $w$ still satisfy the conditions
\eqref{u0,vk=} and \eqref{int-w=0},
In particular, if we let $C^\infty_0(\RR^n$ denote the smooth functions with compact support, we have
\begin{equation}\label{weak-Lu=f}
\int_{\RR^n} a_{ij}(x)\,\partial_i \widetilde u\,\partial_j \eta \,dx=\int_{\RR^n} f_i\,\partial_i\eta\,dx-\int_{\RR^n} f_0\,\eta\,dx
\quad\hbox{for all}\ \eta\in C^\infty_0(\RR^n).
 \end{equation}
 After we find $\widetilde u$ and estimates on $u_0$, $\vec v$, and $w$, we will show $\widetilde u=\chi\,u$.

If we plug \eqref{u-asym} into \eqref{weak-Lu=f} and take $\eta=\eta(r)$, we obtain (cf.\ (55a) in \cite{MM}):
\begin{subequations}
\begin{equation}\label{eta(r)-ODE-a}
\alpha(r)\,u_0'(r)+r\,\vec\beta(r)\cdot\vec v\,'(r)+\vec\gamma(r)\cdot\vec v(r)+p[\nabla w](r)=\widetilde f(r)+r^{1-n}\int_0^r \overline{f_0}(\rho)\rho^{n-1}\,d\rho,
\end{equation}
where
\begin{equation} \label{eta(r)-ODE-b}
\begin{aligned}
\alpha(r)=\meanint_{S^{n-1}} a_{ij}(r\theta)\theta_i\theta_j\,ds_\theta,&
\qquad \beta_k(r)=\meanint_{S^{n-1}} a_{ij}(r\theta)\theta_i\theta_j\theta_k\,ds_\theta,
  \\ 
\gamma_j(r)=\meanint_{S^{n-1}} a_{ij}(r\theta)\theta_i\,ds_\theta,& \qquad
p[\nabla w](r)=\meanint_{S^{n-1}}a_{ij}(r\theta)\,\del_jw(r\theta)\,\theta_i\,ds_\theta, 
\\
\widetilde f(r)=\meanint_{S^{n-1}} f_i(r\theta)\theta_i\,d\theta, & \qquad \overline{f_0}(r)=\meanint_{S^{n-1}} f_0(r\theta)\,d\theta.
\end{aligned}
\end{equation}
Using \eqref{a_ij-delta_ij}, for $0<r<1$ we see that these coefficients satisfy 
\begin{equation} \label{eta(r)-ODE-c}
\begin{aligned}
|\alpha(r)-1|,\ & |\vec\beta(r)|,\ |\vec \gamma(r)|\, \leq \om(r)\, \\
|p[\nabla w](r)|& \leq \om(r)\,\meanint_{S^{n-1}} |\nabla w|\,ds,
\end{aligned}
\end{equation}
\end{subequations}
while for $r>1$ we have $\alpha(r)=1$ and $\beta_k(r)=\gamma_k(r)=p[\nabla w](r)=0$. Without loss of generality, we may assume $\alpha(r)>0$ for all $r>0$.

Similarly, we can take $\eta=\eta(r)x_\ell$ in \eqref{weak-Lu=f} and obtain a second-order liner system of ODEs. But we can use \eqref{eta(r)-ODE-a} to eliminate $u_0$ and then reduce the second-order system for $\vec v$ to a first-order system for $(\vec v,\vec v_r)$; of course these systems also depend on $w$. The first-order system is simplified if we change the independent variable to $t=-\log r\in (-\infty,\infty)$.
This brings the dynamical system \eqref{DS:joint} into play.
If we follow the calculations in \cite{MM}, we find that the first-order system for $(\vec v,\vec v_r)$ that depends on $w$ may be converted to the form
\begin{equation} \label{DynSys+w}
\frac{d}{dt}
\begin{pmatrix}
\phi \\ \psi
\end{pmatrix} + \begin{pmatrix}
0 & 0 \\ 0 & -nI
\end{pmatrix}
\begin{pmatrix}
\phi \\ \psi
\end{pmatrix} + {\mathcal R}(t)
\begin{pmatrix}
\phi \\ \psi
\end{pmatrix} =  F(t,\nabla w)+G(t),
\end{equation}
where ${\mathcal R}(t)=0$ for $t<0$, but for $t>0$ it is of the form \eqref{DS:R} with
\[
R_1(t) = \meanint_{S^{n-1}}\left(A(r\theta)-nA(r\theta)\theta\otimes\theta\right)\,d\theta\,(1+O(\e^2(t))\quad\hbox{as}\ t\to\infty.
\]
The term $F(t,\nabla w)$ vanishes for $t<0$, but 
\begin{equation}\label{F(w)-condition}
|F(t,\nabla w)| \leq c\,\e(t)\,\meanint_{S^{n-1}}|\nabla w|\,d\theta \quad\hbox{for}\ t>0,
\end{equation}
and the term $G(t)$ comes from $\vec f$ and $f_0$, so is supported in $\log 2\leq t\leq 2\log 2$ and we can estimate
\begin{equation}\label{G(t)-condition}
\|{\mathcal E}^{-1}G\|_{L^1}\leq c\,(\|\vec f\|_{L^2}+\|f_0\|_{L^2}).
\end{equation}
Moreover, the relationship between $(\phi,\psi)$ and $(\vec v,\vec v_t)$ satisfies:
 \begin{equation} \label{phi=nv}
 \begin{pmatrix}
\phi \\ \psi
\end{pmatrix}
-\frac{1}{n^2}
 \begin{pmatrix}
n\vec v-\vec v_t \\ \vec v_t
\end{pmatrix}
\leq
c\,\e(t)\left(
|\vec v(t)|+|\vec v_t(t)|+\meanint_{S^{n-1}}|\nabla w|d\theta
\right).
\end{equation}

Now, given $w$ with certain properties, we solve \eqref{DynSys+w} with initial conditions $\phi(0)=0=\psi(0)$ to find $\phi,\psi$ and hence $\vec v$, $\vec v_r$. However, we need to separately control the dependence of $\vec v$ on $w$, so let us write $\vec v=\vec v\,^w+\vec v\,^\circ$ where $\vec v\,^w$ is the solution with $G=0$ and $\vec v\,^\circ$ is the solution with $F=0$. Similarly, when we use \eqref{eta(r)-ODE-a} to solve for $u_0$ in terms of $\vec v$ and $w$, we want to write $u_0=u_0^w+u_0^\circ$ where $u_0^w$ depends  on $w$ (including $\vec v\,^w$), while $u_0^\circ$ is independent of $w$:
\[
\begin{aligned}
u'_0&=(u_0^w)'+(u_0^\circ)'=\alpha^{-1}\left(-r\vec\beta\cdot(\vec v\,^w)'-\vec \gamma\cdot \vec v\,^w-p[\nabla w]\right) \\
&+\alpha^{-1}\left(\widetilde f(r)+r^{1-n}\int_0^r \overline{f_0}(\rho)\rho^{n-1}\,d\rho-r\vec\beta\cdot(\vec v\,^\circ)'-\vec \gamma\cdot \vec v\,^\circ\right)
\end{aligned}
\]

Finally, let us find the PDE that $w$ must satisfy.
As in \cite{MM}, let us introduce the matrix
\begin{equation}
\Omega_{ij}(x)=a_{ij}(x)-\delta_{ij},
\end{equation}
which satisfies $|\Omega_{ij}(x)|\leq \om(|x|)$ for $|x|<1$ and $\Omega(x)=0$ for $|x|>1$. Recalling that $Pw=0$ and the notation \eqref{def:perp}, we find that $w$ satisfies
\[
\Delta w+[{\rm div}(\Omega\nabla w)]^\perp+ [{\rm div}(\Omega\nabla (\vec v\cdot x))]^\perp+
 [{\rm div}(\Omega\nabla (u_0)]^\perp=[\partial_if_i+f_0]^\perp.
\]
But we want to write this as
\begin{equation}\label{Lap(w)=}
\begin{aligned}
\Delta w&+[{\rm div}(\Omega\nabla w)]^\perp+ [{\rm div}(\Omega\nabla (\vec v\,^w\cdot x))]^\perp+
 [{\rm div}(\Omega\nabla (u^w_0)]^\perp= \\
& [\partial_if_i+f_0-{\rm div}(\Omega\nabla (\vec v\,^\circ\cdot x)) - {\rm div}(\Omega\nabla (u_0^\circ)]^\perp.
 \end{aligned}
\end{equation}

If we apply $\Delta^{-1}$ (i.e.\ convolution by the fundamental solution) to both sides of  \eqref{Lap(w)=}, we obtain
\begin{subequations}\label{w+T(w)=eta}
\begin{equation}
w+T_1(w)+T_2(w)+T_3(w)=\xi,
\end{equation}
where 
\begin{equation}
T_1(w)=\Delta^{-1}\left[{\rm div}(\Omega\nabla w)\right]^\perp
\end{equation}
\begin{equation}
T_2(w)=\Delta^{-1}\left[{\rm div}(\Omega\nabla(\vec v^{\,w}\cdot x))\right]^\perp
\end{equation}
\begin{equation}
T_3(w)=\Delta^{-1}\left[{\rm div}(\Omega\nabla(u_0^w))\right]^\perp
\end{equation}
\begin{equation}
\xi=\Delta^{-1}\left[ \partial_if_i+f_0 -{\rm div}(\Omega\nabla (\vec v\,^\circ\cdot x)) - {\rm div}(\Omega\nabla (u^\circ_0)\right]^\perp.
\end{equation}
\end{subequations}
We want to show that \eqref{w+T(w)=eta} admits a solution $w\in Y$, where $Y$ denotes those functions in $y\in H^{1,2}(\RR^n\backslash\{0\})$ for which the following norm is finite:
\begin{equation}
\|y\|_Y=\sup_{0<r<1}\frac{M_{1,2}(y,r)}{\om(r)\,r\,E(r)}+\sup_{r>1}\frac{M_{1,2}(y,r)}{\delta\,r^{-n}}.
\end{equation}
Here $\delta>0$ will be small, but  we may still assume $\e(t)$ satisfies \eqref{small_epsilon}. For $\om(r)$ this means
$\om(r)<\delta $ and $\int_0^1 \om^2(r)\,r^{-1}\,dr<\delta$ \ for $ 0<r<1$.

We need to show $\xi\in Y$ and that each of the mappings $T_j:Y\to Y$ is small.
For future reference, we note that $y\in Y$ implies
\begin{equation}\label{est:Mp(grad(y))}
M_2(\nabla y,r)\leq \begin{cases}
\om(r)\,E(r)\,\|y\|_Y & 0<r<1 \\
 \delta\,r^{-n-1}\,\|y\|_Y & r>1.
\end{cases}
\end{equation}

First let us confirm that $T_1:Y\to Y$ with small norm. Let $y\in Y$ with $\|y\|_Y\leq 1$. Use Proposition \ref{pr:3} with $\vec f=\Omega\nabla y$ to conclude
\[
M_{1,2}(T_1(y),r)\leq c\left(r^{-n}\int_0^r M_2(\Omega\nabla y,\rho)\,\rho^n\,d\rho+r^2\int_r^\infty M_2(\Omega\nabla y,\rho)\,\rho^{-2}\,d\rho\right).
\]
For any function $y\in H^{1,2}_{\loc}(\RR^n\backslash\{0\})$, note that $M_2(\Omega\nabla y,r)=0$ for $r>1$ and $M_2(\Omega\nabla y,r)\leq \om(r)\,M_2(\nabla y,r)$ for $0<r<1$; but for $y\in Y$ with $\|y\|_Y\leq 1$ this last estimate becomes $M_2(\Omega\nabla y,r) \leq \om^2(r)\,E(r)$. Consequently, for $0<r<1$ we have
\begin{equation}\label{est:M_(1,p)(T1)}
M_{1,2}(T_1(y),r)
\leq c\left(r^{-n}\int_0^r \om^2(\rho)\,E(\rho)\rho^n\,d\rho+r^2\int_r^1 \om^2(\rho)\,E(\rho)\rho^{-2}\,d\rho\right).
\end{equation}
Similar to \eqref{int-exp*E}, we can integrate by parts to show
\begin{equation}\label{int-exp*E(r)}
r^{-n}\int_0^r E(\rho)\rho^{n}\,d\rho\leq c\,r\,E(r).
\end{equation}
We can use this, the monotonicity of $\om$, and \eqref{small_epsilon} to estimate the first integral in \eqref{est:M_(1,p)(T1)}:
\[
r^{-n}\int_0^r \om^2(\rho)\,E(\rho)\rho^n\,d\rho \leq c\,\delta\, \om(r)\,r\,E(r).
\]
We can also use \eqref{om-notrapiddecrease} and \eqref{E*r^lambda-decreasing} to estimate  the second integral in \eqref{est:M_(1,p)(T1)}:
\[
r^2\int_r^1 \om^2(\rho)\,E(\rho)\rho^{-2}\,d\rho\leq \delta\,r^2\om(r)r^{-1+\kappa}E(r)\,r^{-\lambda}\int_r^1\rho^{-1-\kappa-\lambda}\,d\rho,
\]
where we have chosen $0<\lambda<\kappa$. We can evaluate this last integral to obtain
\[
r^2\int_r^1 \om^2(\rho)\,E(\rho)\rho^{-2}\,d\rho\leq \frac{\delta}{\kappa-\lambda}\,\om(r)\,r\,E(r)(1+r^{\kappa-\lambda})
\leq c\,\delta\,\om(r)\,r\,E(r).
\]
Meanwhile for $r>1$, we have 
\[
M_{1,2}(T_1(y),r)\leq c\,r^{-n}\int_0^1 M_2(\Omega\nabla y,\rho)\rho^n\,d\rho\leq c\,\delta^2 \,r^{-n}\int_0^1 E(\rho)\rho^n\,d\rho
\leq c\,\delta^2 \,r^{-n}.
\]
These estimates together show $\|T_1(y)\|_{Y}\leq c\,\delta$. Taking $\delta$ small enough, we have $\|T_1\|_{Y\to Y}<1/3$.

For $T_2$ we need to use Proposition \ref{pr:2} to obtain estimates on $\vec v\,^y$ for $y\in Y$, so let us first confirm that 
$F(t,\nabla y)$ in \eqref{DynSys+w} satisfies the conditions \eqref{DS:F1} and \eqref{DS:F2}. For  \eqref{DS:F1}  we use  \eqref{F(w)-condition},  to compute
\[
\begin{aligned} 
\int_0^\infty {\mathcal E}^{-1}(t)\,|F_1(t,\nabla y)|dt &\leq \int_0^\infty  {\mathcal E}^{-1}(t)\,\e(t)\,\meanint |\nabla y|\,d\theta\,dt \\
&=c\int_0^1 E^{-1}(r)\,\om(r)\,\meanint |\nabla y|\,d\theta\,\frac{dr}{r} \\
&=c\sum_{j=0}^\infty \int_{2^{-j-1}}^{2^{-j}} E^{-1}(r)\,\om(r)\,\meanint|\nabla y|\,d\theta\,\frac{dr}{r} 
\end{aligned} 
\]
But  \eqref{om-notrapiddecrease} and \eqref{E*r^lambda-decreasing} imply that for $2^{-j-1}<r<2^{-j}$ we have $E^{-1}(r)<2^\lambda\,E^{-1}(2^{-j-1})$ and $\om(r)<2^{1-\kappa}\om(2^{-j-1})$, so
\[
\begin{aligned} 
\int_0^\infty {\mathcal E}^{-1}(t)\,|F_1(t,\nabla y)|dt 
&\leq c\,\sum_{j=0}^\infty E^{-1}(2^{-j-1})\,\om(2^{-j-1})\,M_1(\nabla y,2^{-j-1}) \\
&\leq c\,\sum_{j=0}^\infty E^{-1}(2^{-j-1})\,\om(2^{-j-1})\,M_2(\nabla y,2^{-j-1}) \\
&\leq c\,\sum_{j=0}^\infty \om^2(2^{-j-1})\,\|y\|_Y \leq c\,\left(\int_0^1\frac{\om^2(r)}{r}dr\right)\,\|y\|_Y.
\end{aligned} 
\] 
Hence we have confirmed \eqref{DS:F1} with
\begin{equation}\label{est:L1-norm(E^(-1)F1}
\|{\mathcal E}^{-1}(t)\,F_1(t,\nabla y)\|_{L^1(0,\infty)}\leq c\,\delta\,\|y\|_Y.
\end{equation}
To confirm \eqref{DS:F2} and estimate $c_\alpha$ for $\alpha\in [n-\delta,n)$, we use the same tools to compute
\[ 
\begin{aligned}
e^{\alpha t}\int_t^\infty |F(t,\nabla y)|e^{-\alpha s}\,ds&=r^{-\alpha}\int_0^r\om(\rho)\meanint |\nabla y|d\theta\,\rho^\alpha\,\frac{d\rho}{\rho}\\
&\leq c\,r^{-\alpha}\,\om(r)\sum_{j=0}^\infty \int_{2^{-j-1}r}^{2^{-j}r}\meanint |\nabla y|\,d\theta\, \rho^\alpha\, \frac{d\rho}{\rho} \\
&\leq c\,r^{-\alpha}\,\om(r)\sum_{j=0}^\infty M_1(\nabla y,2^{-j-1}r)\,(2^{-j}r)^\alpha \\
&\leq c\,r^{-\alpha}\,\om(r)\sum_{j=0}^\infty M_2(\nabla y,2^{-j-1}r)\,(2^{-j}r)^\alpha \\
&\leq c\,r^{-\alpha}\,\om(r)\,\|y\|_Y\,\sum_{j=0}^\infty \om(2^{-j-1}r)\,E(2^{-j-1}r)\,(2^{-j-1}r)^\alpha  \\
&\leq c\,r^{-\alpha}\,\om(r)\,\|y\|_Y\int_0^r \om(\rho)\,E(\rho)\,\rho^\alpha\,\frac{d\rho}{\rho}\\
&\leq c\,r^{-\alpha}\,\om^2(r)\,E(r)\,r^{\alpha/2}\,\|y\|_Y\int_0^r \rho^{\alpha/2-1}\,{d\rho}\\
&= c\,\om^2(r)\,{E}(r)\,\|y\|_Y=c\,\e^2(t)\,{\mathcal E}(t)\,\|y\|_Y.
\end{aligned}
\]
We conclude that  \eqref{DS:F2} holds with
\begin{equation}\label{est:c-alpha}
c_\alpha=c\,\delta\,\|y\|_Y.
\end{equation}

Now let us confirm  $T_2:Y\to Y$ with small norm. We want to use Proposition \ref{pr:3} with $\vec f=\Omega\nabla(\vec v\,^y\cdot x)$, so we need to estimate $M_2(\nabla(\vec v\,^y\cdot x))$ for $|x|<1$. Using \eqref{phi=nv} and then Proposition \ref{pr:2} with \eqref{est:L1-norm(E^(-1)F1} and \eqref{est:c-alpha} we have
\[
\begin{aligned}
|\nabla(\vec v\,^y\cdot x))|&\leq c(r|(\vec v\,^y)'|+|\vec v\,^y|)\leq c\,(1+\e(t))\,(|\phi^y(t)|+|\psi^y(t)|) +c\,\e(t)\meanint|\nabla y|\,d\theta\\
&  \leq c\,\delta\,{\mathcal E}(t)\,\|y\|_Y= c\,\delta\, E(r)\,\|y\|_Y.
\end{aligned}
\]
Assuming $\|y\|_Y\leq 1$ and applying Proposition \ref{pr:3}, we have for $0<r<1$
\[
\begin{aligned}
M_{1,2}(T_2[y],r)&\leq c\,\delta\left(r^{-n}\int_0^r \om(\rho)E(\rho)\rho^n\,d\rho
+r^2\int_r^1 \om(\rho)E(\rho)\,\rho^{-2}\,d\rho\right) \\
&\leq c\,\delta\,\om(r)\,r\,E(r)
\end{aligned}
\]
where we have used \eqref{int-exp*E(r)}, \eqref{om-notrapiddecrease}, and \eqref{E*r^lambda-decreasing} as we did for $T_1$.
Meanwhile, for $r>1$ we have
\[
M_{1,2}(T_2[y],r)\leq c\,\delta^2\,r^{-n}\int_0^1E(\rho)\,\rho^n\,d\rho \leq c\,\delta^2\,r^{-n}.
\]
So $\|T_2[y]\|_Y<c\,\delta$ and by taking $\delta$ small enough we can arrange $\|T_2\|_{Y\to Y}<1/3$.

The last map to consider is $T_3$, so we need to estimate $M_2(\nabla u_0^y,r)$ for $0<r<1$. Recall that $u_0^y$ comes from 
\eqref{eta(r)-ODE-a} without $\vec f$ and $f_0$, so
\[
|\nabla u_0^y(r)|\leq c\left( \om(r)\,r\,|(\vec v\,^y)'(r)|+\om(r)|\vec v\,^y(r)|+p[\nabla w](r)\right)
\leq c\,\delta\,\om(r)\,E(r)\,\|y\|_Y.
\]
The rest of the proof that $\|T_3\|_{Y\to Y}<1/3$ is the same as for $T_2$.
 
Finally we need to confirm that $\xi\in Y$. We take the terms one at a time. For $\xi_1=\Delta^{-1}[{\rm div}\vec f]^\perp$ where $\vec f\in L^2(\RR^n)$ has support in $1/4<|x|<1/2$, we apply Proposition 3b to conclude
\[
M_{1,2}(\xi_1,r)\leq \begin{cases} c\,\|\vec f\|_{L^2}\,r^2 & 0<r<1 \\ c\,\|\vec f\|_{L^2}\,r^{-n} & r>1.\end{cases}
\]
Similarly, for $\xi_2=\Delta^{-1}[ f_0]^\perp$ where $ f_0\in L^2(\RR^n)$ has support in $1/4<|x|<1/2$,  we apply Proposition 3a to conclude
\[
M_{1,2}(\xi_2,r)\leq M_{2,2}(\xi_2,r)\leq \begin{cases} c\,\| f_0\|_{L^2}\,r^2 & 0<r<1 \\ c\,\| f_0\|_{L^2}\,r^{-n} & r>1.\end{cases}
\]
For $\xi_3=\Delta^{-1}[ {\rm div}(\Omega\nabla(\vec v\,^\circ \cdot x)]^\perp$, we need to estimate $M_2(\nabla(\vec v\,^\circ \cdot x),r)$ for $0<r<1$.
But by Proposition 2 applied to \eqref{DynSys+w} with $F=0$ and $G$ satisfying \eqref{G(t)-condition}, we know that
\[
|\nabla(\vec v\,^\circ\cdot x)|\leq c\,(|\phi(t)|+|\psi(t)|) \leq c\,(\|\vec f\|_{L^2}+\|f_0\|_{L^2}).
\]
This enables us to estimate (using tricks as above)
\[
\begin{aligned}
M_{1,2}(\xi_3,r)&\leq c\,(\|\vec f\|_{L^2}+\|f_0\|_{L^2})\left(r^{-n}\int_0^r \om(\rho)\,E(\rho)\,\rho^n\,d\rho+r^2\int_r^1\om(\rho)\,E(\rho)\,\rho^{-2}\,d\rho\right) \\
&\leq c\,\begin{cases} (\|\vec f\|_{L^2}+\|f_0\|_{L^2})\,\om(r)\,r\,E(r) & 0<r<1, \\
 (\|\vec f\|_{L^2}+\|f_0\|_{L^2})\, r^{-n} & r>1. \end{cases}
\end{aligned}
\]
Finally, for $\xi_4=\Delta^{-1} [{\rm div}(\Omega\nabla (u^\circ_0)]^\perp$ we need to estimate  $M_2(\nabla (u^\circ_0),r)$ for $0<r<1$.
But from \eqref{eta(r)-ODE-a} we know
\[
(u_0^\circ)'(r)=\alpha^{-1}(r)\left[\widetilde f(r)+r^{1-n}\int_0^r \overline{f_0}(\rho)\,d\rho\right],
\]
which has support in $1/4<r<1/2$, so $M_{1,2}(\xi_4,r)$ satisfies the same estimates as for $M_{1,2}(\xi_1,r)$ and $M_{1,2}(\xi_1,r)$.
We summarize this by
\begin{equation}\label{est:|xi|_Y}
\|\xi\|_Y\leq c\,(\|\vec f\|_{L^2}+\| f_0\|_{L^2})\leq c\,\|u\|_{L^2(B_1)}.
\end{equation}

We have shown that there exists $w\in Y$ with $\|w\|_Y\leq c\,\|u\|_{L^2(B_1)}$ such that the function
 $\widetilde u(x)$ as in \eqref{tilde-u-asym} is a weak solution of \eqref{Lu=f}. Moreover, $\widetilde u$ satisfies the desired estimate \eqref{grad(u)-est} 
 since $M_{2}(u'_0,r)$, $M_{2}(r\,\vec v,'r)$, $M_2(\nabla w,r)\leq c\,\om(r)\,E(r)\,\|u\|_{L^2(B_1)}$ and $M_2(\vec v,r)\leq  c\,E(r)\,\|u\|_{L^2(B_1)}$.
 It only remains to show that $\widetilde u=\chi u$. But  $z=\widetilde u-\chi u$ is a solution of $\partial_j(a_{ij}\partial_i u)=0$ in $\RR^n$ that vanishes at infinity, so $z=0$ by the maximum principle.  
 $\Box$

\medskip\noindent
  \section{The Gilbarg-Serrin Example}\label{sec:GS}

\medskip
   In \cite{GS}, Gilbarg and Serrin considered coefficients of the form
 \begin{equation}\label{GS}
 a_{ij}(r,\theta)=\delta_{ij}+g(r)\,\theta_i\theta_j, \quad\hbox{where $r=|x|$ and $\theta_i=x_i/r$.}
 \end{equation}
They assumed $g(r)$ is a bounded function  satisfying $g(r)>-1+\varepsilon$, which guarantees uniform ellipticity.
In \cite{MM} it was shown that the system \eqref{DynSys-R1} reduces to the scalar equation 
\begin{equation}\label{eq:ODE}
\frac{d\phi}{dt}=\frac{n-1}{n}\, {\rm g}(t)\,\phi.
 \end{equation}
There is no matrix, so the unique ``eigenvalue'' of $-R(r)$ is $\frac{n-1}{n}\,  g(r)$ and
 \[
E(r)=\exp\left[\frac{n-1}{n}\int_r^1\frac{g(\rho)}{\rho}\,d\rho\right].
 \]
In \cite{MM2} it was shown that, provided $g(r)$ has finite variation on $(0,1)$, then there is a solution of \eqref{Lu=0} of the form
 \eqref{u-asym} with $u_0=0$,
 \begin{equation}
\vec v(r)=c\,\vec e_j \exp\left[\frac{n-1}{n}\int_r^1\frac{g(\rho)}{\rho}\,d\rho\right]\left(1+o(1))\right) \quad\hbox{as}\ r\to 0,
\end{equation}
and $w$ satisfying $M_2(\nabla w,r)\leq c\,\om(r)\,{E}(r)$.
Here $\vec e_j$ is any one of the $n$ coordinate unit vectors, so we have $|\vec v|=c\,E(r)(1+o(1))$
and we see that the estimate \eqref{grad(u)-est} in Theorem 1 is sharp in this case.

  \bigskip\noindent
{\bf Acknowledgement}: This paper has been supported by the RUDN University Strategic
Academic Leadership Program.


\end{document}